\documentclass[a4paper,11pt]{amsart}

\usepackage{amsmath,amsfonts,amssymb,graphics}
\usepackage{amsthm}
\usepackage{pgfplots}
\usepackage{cite}
\usepackage{hyperref}
\usepackage{color}
\usepackage[normalem]{ulem}
\usepackage{enumitem}
\usepackage{bbm}
\usepackage{multirow}
\usepackage{booktabs}
\usepackage[foot]{amsaddr}

\newtheorem{theorem}{Theorem}[section]
\newtheorem{lemma}[theorem]{Lemma}
\newtheorem{proposition}[theorem]{Proposition}
\newtheorem{corollary}[theorem]{Corollary}

\newtheorem{conjecture}[theorem]{Conjecture}

\theoremstyle{remark}
\newtheorem{remark}[theorem]{Remark}

\numberwithin{equation}{section}

\newcommand {\R} {\mathbb{R}}
\newcommand {\Z} {\mathbb{Z}}
\newcommand{\F} {\mathcal{F}}

\renewcommand {\P} {\mathbb{P}}

\renewcommand {\S} {\mathcal{S}}

\newcommand {\E} {\mathbb{E}}

\newcommand {\Var} {\mathrm{Var}\,}

\renewcommand{\d}{\text{d}}

\def\br#1{\left(#1\right)}

\graphicspath{{./pics/}}
\begin{document}

\title[]{Coupling of stationary fields with application to arithmetic waves}
\author{Dmitry Beliaev\textsuperscript{1}}
\address{\textsuperscript{1}Mathematical Institute, University of Oxford}
\email{belyaev@maths.ox.ac.uk}
\author{Riccardo W. Maffucci\textsuperscript{1,2}}
\address{\textsuperscript{2}Current address: EPFL}
\email{riccardo.maffucci@epfl.ch}
\thanks{Both authors were supported by the Engineering \& Physical Sciences Research Council (EPSRC) Fellowship EP/M002896/1.
\\
R. M. was supported by Swiss National
Science Foundation project 200021\_184927}

\date{\today}
\begin{abstract}
In this paper we obtain a range of quantitative results of the following type: given two centered Gaussian fields with close covariance kernels we construct a coupling such that the fields are uniformly close on some compact with probability very close to one. As an application we show that it is possible to couple arithmetic random waves so that they converge locally uniformly to the random plane wave and estimate the rate of convergence. 

\end{abstract}
\maketitle

\section{Introduction}

In many problems involving Gaussian fields one quite often needs a quantitative statement of the following kind: if two fields have similar structure, then it is possible to couple them in such a way that they are close to each other with very high probability. There are many different custom-made statements of this type that depend on a particular notion of `similarity' and `closeness'.

Most of these results treat `similarity' in terms of the covariance kernel. One of the main claims of this paper is that in many cases it is more productive to think in terms of the spectral measure. This helps to treat the cases of singular spectral measures which is hard to analyse in terms of kernels and not amenable to white noise decomposition methods. 

Before discussing the results, we would like to explain what do we mean by the `closeness'. There are many ways to define it in the context of Gaussian fields. We are interested in uniform estimates that are naturally required for many analytic techniques. To be more precise, we are interested in results of the following type: given two $C^k$ smooth fields $f_1$ and $f_2$ we want to couple them in such a way that for a given domain $\Omega$, the norm $\|f_1-f_2\|_{C^k(\Omega)}$ is small with large probability. Below we will discuss various ways to quantify this statement.


The paper is organized as follows: in Section \ref{s: background} we introduce some necessary background, namely we discuss the white noise representation of Gaussian fields and the large deviation results for the $C^k$ norm of a field. In Section \ref{s: smooth case} we show that if spectral measures of two fields are continuous and their densities are close in a certain sense, then the fields can be coupled in such a way that their difference is a \emph{stationary} Gaussian field with small variance. The variance is given by Theorem \ref{thm:spectral density}. Combining this with Corollary \ref{cor: norm of the field} we obtain a quantified result about the $C^k$ norm of the difference.  In Section \ref{s: general case} we deal with the case when the spectral measures are singular with respect to each other. In particular, it means that they are not close in the sense mentioned above, so the results of Section \ref{s: smooth case} are not applicable. In this case it is possible to formulate a different notion of closeness of the spectral measure. We show that if the spectral measures are sufficiently close, then the fields can be coupled in such a way that their difference is a \emph{non-stationary} field such that its variance is small in a large ball. Variance of this coupling is estimated in Theorem \ref{thm: general coupling}. Finally, in Section \ref{s: application} we introduce a particularly important family of Gaussian fields: random arithmetic waves. In this section we quantify how they converge to the limiting random plane wave. This application was the main motivation for this paper. 

\subsection{Acknowledgement} We are very grateful to Alejandro Rivera for helpful conversations. 

\section{Background}
\label{s: background}

\subsection{White noise representation}
In this note we are interested in a particular setup which appears quite often: smooth stationary fields on $\R^n$. There are two standard ways to describe such fields. Usually, a stationary field is described by its covariance kernel 
\[
K(x)=\E[f(0)f(x)]=\E[f(y)f(y+x)].
\]
Alternatively, by Bochner's theorem, there is a positive symmetric measure $\rho$, called the \emph{spectral measure}, such that
\[
K(x)=\int_{\R^n}e^{2\pi i x\cdot t} \d\rho( t). 
\]
Here, and everywhere below, we assume that $f$ is $k$ times differentiable, $K$ is $2(k+1)$ times differentiable, and $\rho$ has finite moments of order up to $2(k+1)$.

The spectral measure gives the alternative characterization of the field. Moreover, the field can be explicitly constructed in terms of the spectral measure. Roughly speaking, the field is the Fourier transform of the white noise with respect to the spectral measure. Here we give a brief summary of the constructions and refer to \cite[A.12]{NS16}.

Let $L^2_H(\rho)$ be the space of all complex-valued Hermitian functions $h$ (i.e. $h(-x)=\bar{h}(x)$) in $\R^n$ such that $\int |h^2|\d \rho<\infty$ (this space is over $\R$). This is a Hilbert space with the usual scalar product $\langle h_1,h_2\rangle=\int h_1 \bar{h}_2 \d \rho$. Note that the symmetry condition implies that the scalar product is always real. The Fourier transform of $L^2_H(\rho)$ is the Hilbert space $\mathcal{H}_K$ such that $K_x(y)=K(x-y)$ is its reproducing kernel. Given an orthonormal basis $\phi_k(x)$ in $\mathcal{H}_K$, the field can be written as $\sum a_k \phi_k(x)$, where $a_k$ are i.i.d. $N(0,1)$ random variables. This series converges locally uniformly, together with its first $k$ derivatives (see \cite[A.5]{NS16}).

The series above can be thought of as the white noise in $\mathcal{H}_K$. The white noise in a Hilbert space $H$ is a collection of centered normal variables indexed by functions from the space (denoted $W(h)$) such that $\E W(h_1)W(h_2)=\langle h_1,h_2\rangle_H$.  The white noise can be informally thought of as $\sum a_k \phi_k$ where $a_k$ are i.i.d. $N(0,1)$ random variables and $\phi_k$ form an orthonormal basis in $H$. Note that in the case of $\mathcal{H}_K$ this series converges locally uniformly.

Since, by construction, the Hilbert structure in $\mathcal{H}_K$ is inherited from $L^2_H(\rho)$, the white noise in $\mathcal{H}_K$ is the Fourier transform of the white noise in $L^2_H(\rho)$, which can be written as $\sum a_k \psi_k$ where $a_k$ as above and $\{\psi_k\}$ is an o.n.b. in $L^2_H(\rho)$. Note that this series is just a formal expression, since, unlike the case above, it does not converge point-wise and should be understood in distributional sense. For this reason, it is more constructive to define the white noise in $L^2_H(\rho)$ as a collection or random variables $W(h)$, $h\in L^2_H(\rho)$. 

With these notations, we can write the stationary field in terms of the white noise in $L^2_H(\rho)$
\begin{equation}
\label{eq:field white noise}
\begin{aligned}
f(x)= W(e^{2\pi i x\cdot t})
\end{aligned}
\end{equation}
which is well defined since $h(t)=\exp(2\pi i x\cdot t)\in L^2_H(\rho)$ for any $x\in\R^n$.

\subsection{$C^k$ norm of a field}

In this paper we will provide several results of the following type: given two $C^k$-smooth centered fields such that their spectral measures are close in some sense, there is a jointly Gaussian coupling of the fields, such that the $C^k$ norm of the difference is small. Note, that if the fields are jointly Gaussian, their difference is a centered Gaussian field. This means that we need estimates of the supremum of a Gaussian field and its derivatives. 


This is a rather standard problem. Roughly speaking, at each point the probability that the field is large has Gaussian tails, but the supremum also depends on how fast the field oscillates. The precise estimate is given by he following lemma:

\begin{lemma}[{cf. \cite[Lemma 3.12]{MV}}]  
\label{l:Borell-TIS}Let $\Omega\subset \R^n$ be a domain that can be covered by $N$ discs of radius $1$. Let $f$ be a centered Gaussian field defined in $V$ -- the $1$-neighborhood  of $\Omega$. Let $K(x,y)=\E f(x)f(y)$ be the covariance kernel of $f$. We assume that $K$ is $k+1$ times differentiable with respect to each variable. We define 
\[
\sigma^2=\sup_{x\in V} \max_{|\alpha|\le k+1} \E (\partial^\alpha f(x))^2=
\sup_{x\in V} \sup_{|\alpha|\le k+1} \partial^\alpha_x\partial^\alpha_y K(x,y)|_{y=x}.
\]

There is a universal constant $c_1>0$ (which only depends on the dimension $n$ and the fact that we use unit balls) such that, for every $A>c_1$ we have that
\begin{equation}
\label{eq: general probability bound}
\P\br{\|f\|_{C^k(\Omega)}>A\sigma}\le \exp\br{\log N-(A-c_1)^2/2}.
\end{equation}
In particular, for large $A$ the upper bound is essentially $\exp(-A^2/2)$.
\end{lemma}

\begin{proof}
By the quantified version of the Kolmogorov theorem \cite[A.9]{NS16} there is an absolute constant $c_1$ such that
\[
\|f\|_{C^k(B(x,1))}=\sup_{|x-y|\le 1}\sup_{|\alpha|\le k} |\partial^\alpha f(y)|\le c_1 \sigma, \quad \forall x\in \Omega.
\]
Applying the Borell-TIS inequality to each partial derivative of order up to $k$ and using the union bound we obtain
\[
\P\br{\|f\|_{C^k(B(x,1))}>c_1\sigma+u}\le c_{k,d}e^{-u^2/2c_1^2\sigma^2},
\]
where $c_{k,d}=1+d+\dots+d^k$ is the number of partial derivatives (we ignore that some of them are equal). Covering $\Omega$ by $N$ discs of radius $1$ and using the union bound we have
\[
\P\br{\|f\|_{C^k(\Omega)}>c_1\sigma+u}\le c_{k,d}Ne^{-u^2/2c_1^2\sigma^2}.
\]
Taking $u=(A-c_1)\sigma$ we have
\[
\P\br{\|f\|_{C^k(\Omega)}>A\sigma}\le \exp\br{\log(c_{k,d} N)-(A-c_1)^2/2c_1^2}.
\]
\end{proof}

For a field defined in the entire space we immediately have the following corollary.
\begin{corollary}
\label{cor: norm of the field}
Let $f$ be a Gaussian field in $\R^n$ and let 
\begin{equation}
\label{eq: sigma_R}
\sigma^2_R=\sup_{|x|\le R+1}\max_{|\alpha|\le k+1}\E(\partial^\alpha f(0))^2.
\end{equation}
Then there is an absolute constant $c_2=c_2(n)$ such that 
\[
\P\br{\|f\|_{C^k(B_R)}>c\sigma_R\log R}\le \exp\br{-(c\log R)^2/16}
\]
assuming that $c\log R>2c_1$, where $c_1$ as in Lemma \ref{l:Borell-TIS}, and $R>c_2$.
\end{corollary}
\begin{proof}
A disc or radius $R+1$ can be covered by $N=c_2 R^d$ discs of radius $1$, where $c_2$ is a constant which depends only on the dimension $d$. Taking this $N$ and $A=c\log R$ with $c>1$, the exponent on the right hand side of \eqref{eq: general probability bound} becomes
\[
\log(c_2R^d)-(c\log R-c_1)^2/2<-(c\log R)^2/16,
\]
provided that $c\log R >2c_1$ and $R$ is sufficiently large. 
\end{proof}
\begin{remark}
The factor $16$ in the statement of Corollary \ref{cor: norm of the field} is not optimal. Assuming that $c\log R$ is much larger than $c_1$ and taking $c_2$ large enough this constant could be reduced to any number greater than $2$. 
\end{remark}
\begin{remark}
This Corollary is particularly useful in the case of a stationary field, since in this case $\sigma_R$ does not depend on $R$.
\end{remark}
\begin{remark}
In exactly the same way, we can use $A=c\log^{1/2}R$. Provided that $c>2\sqrt{n}$ and $R$ is sufficiently large we get that
\[
\P\br{\|f\|_{C^k(B_R)}>c\sigma\log^{1/2} R}\le R^{n-c^2/4}.
\]
This means that the norm is of order $\sigma\sqrt{\log R}$, but on this scale we only have polynomial decay of tails. To have a super-polynomial decay we need $A$ to be of order $\log R$. In the same way, by taking $A$ to be of order $\sqrt{R}$ we get an exponential decay.
\end{remark}

Lemma \ref{l:Borell-TIS} and its Corollary \ref{cor: norm of the field} describe tails of the field in terms of the parameter $\sigma$. In the remaining part of the paper we will consider several couplings of Gaussian fields such that their difference is also a Gaussian field. The results above give us that in order to estimate how small is the difference in $C^k(B_R)$ norm we just have to compute the corresponding $\sigma^2_R$ for the difference of fields.


\section{Smooth case}
\label{s: smooth case}

Let us first consider the simplest case of two fields that could be thought of as a perturbation of the same field. This is a well known case and we present it here only in order to give a simple example which shows what kind of couplings and results we are aiming at.

 Let us assume that the fields $f_1$ and $f_2$ have covariance kernels $K_j=(1-\epsilon_j)K+\epsilon_j L_j$, where $K$ and $L_j$ are positive definite functions normalized $K(0)=L_j(0)=1$. In this case it is very natural to write 
\[
f_j(x)=\sqrt{1-\epsilon_j}f(x)+\sqrt{\epsilon_j} g_j(x),
\]
where $f$ and $g_i$ are independent fields with covariance kernels $K$ and $L_j$ correspondingly.
For this coupling we have that $f_1-f_2$ is a Gaussian field with covariance (and its derivatives) of order $\epsilon_1+\epsilon_2$. Assuming that both $\epsilon_j$ are of the same order $\epsilon$ we have that $\sigma^2_R$ from Corollary \ref{cor: norm of the field} is of order $\epsilon$.  

\begin{remark}
This is the first obvious step in the right direction, but it is clear that this is sub-optimal. Even in the case of coupling just two normal variables, there is a better coupling. If we have two central normal variables $X_1$ and $X_2$ with variances $\sigma_1^2$ and $\sigma_2^2$, then our coupling is morally equivalent to writing $X_2=X_1+X$, where $X\sim N(0,\sigma_2^2-\sigma_1^2)$ (assuming $\sigma_2>\sigma_1$). In this case the difference has variance $\sigma_2^2-\sigma_1^2$. If, as above, we denote by $\epsilon$ the difference between variances, we have that the difference is typically of order $\sqrt{\epsilon}$. On the other hand there is an $L^2$ optimal coupling $X_2=(\sigma_2/\sigma_1)X_1$. Alternatively, we can write $X_i=\sigma_i X$ where $X\sim N(0,1)$. In this case $X_2-X_1\sim N(0,(\sigma_1-\sigma_2)^2)$ which has variance of smaller order, the typical difference is of order $\epsilon$. Moreover, it is not hard to show that there is no coupling (even not jointly Gaussian) such that tails of $X_1-X_2$ decay faster.  
\end{remark}

Next, we consider the case when the spectral measures are sufficiently smooth and covariance kernels decay sufficiently fast. In particular, in this section we assume that the spectral measure has a density, that is there is a positive symmetric function $\rho$ such that $\d\rho(t)=\rho(t)\d t$. In this case there is a simple connection between the white noises in $L^2_H(\rho)$ and in $L^2_H(\d t)$. Namely,
\[
W_{L^2_H(\rho)}(f)=W_{L^2_H(\d t)}(f\sqrt{\rho}).
\] 
Alternatively, in the distributional sense, we can write 
\[
W_{L^2_H(\rho)}=\sqrt{\rho}W_{L^2_H(\d t)}.
\]
This means that the field can be written as
\[
f(x)=W_{L^2_H(\d t)}\br{\sqrt{\rho(t)}e^{2\pi i x\cdot t}}.
\]

Given two different fields with spectral densities $\rho_1$ and $\rho_2$, we can use the formula above to couple the two fields, namely, we use the same white noise to define both fields. 

\begin{remark}
The Fourier transform of $L^2_H(\d t)$ is the real $L^2(\R^n)=L^2(\R^n,\d x)$, so we can rewrite
\[
f(x)=\F W_{L^2_H(\rho)}=\F\br{ \sqrt{\rho}W_{L^2_H(\d t)}}=W*q=W(q(x-\cdot))
\]
where $q=\F(\sqrt{\rho})$ is the convolution square root of $K$, namely $K=q*q$ and $W$ is the real white noise in the standard $L^2(\R^n)$. This representation of stationary Gaussian fields is rather standard in statistics. In the context of the coupling of fields it was used in \cite{MV}. 
\end{remark}

\begin{remark}
This construction is very similar to the aforementioned optimal $L^2$ coupling of normal random variables. There is a similar optimal coupling of multivariate normal random  variables. In this coupling we write both random variables as linear transformations of the same standard normal. In our case, we write both fields as transformations of the same white noise.  
\end{remark}

With this coupling, the difference between fields becomes 
\[
F(x)=f_1(x)-f_2(x)=W_{L^2_H(\d t)}\br{\br{\sqrt{\rho_1(t)}-\sqrt{\rho_2(t)}}e^{2\pi i x\cdot t}}.
\]
This is a \emph{stationary} Gaussian field with spectral density $(\sqrt{\rho_2}-\sqrt{\rho_1})^2$. Its variance at any point is 
\[
\E F(x)^2=\E F(0)^2=\int \br{\sqrt{\rho_1(t)}-\sqrt{\rho_2(t)}}^2\d t=\|\sqrt{\rho_2}-\sqrt{\rho_1}\|_2^2.
\]
Similarly, for a multi-index $\alpha$ with $|\alpha|\le k+1$ we have
\[
\begin{aligned}
\E(\partial^\alpha F(0))^2=&(2\pi)^{2|\alpha|}\int t^{2\alpha}(\sqrt{\rho_2}-\sqrt{\rho_1})^2\d t \\
=&(2\pi)^{2|\alpha|}\|t^\alpha (\sqrt{\rho_2}-\sqrt{\rho_1})(t)\|_2^2.
\end{aligned}
\]
This argument proves the following theorem.
\begin{theorem}
\label{thm:spectral density}
Let $f_1$ and $f_2$ be two centered stationary Gaussian fields with spectral densities $\rho_1$ and $\rho_2$. We assume that both spectral measures have finite moments of order up to $2(k+1)$ for some integer $k\ge 0$. Then there is a coupling of these fields such that $F=f_2-f_1$ is a centered  stationary Gaussian field with 
\[
\sigma^2=\sup_{|\alpha|\le k+1}(2\pi)^{2|\alpha|}\|t^\alpha (\sqrt{\rho_2}-\sqrt{\rho_1})(t)\|_2^2,
\]
where $\sigma^2$ is defined by \eqref{eq: sigma_R} (it does not depend on $R$ in this case). 
\end{theorem}
Corollary \ref{cor: norm of the field} gives us estimates on the tails of this coupling.

We formulate the result in terms of measures having density with respect to the Lebesgue measure since this is by far the most important case. We would like to point out that exactly the same argument works if $\d \rho_j(t)=\rho_j(t)\d\mu(t)$ where $\mu$ is a symmetric measure and $\rho_j$ are even functions from $L^1(\d\mu)$. In this case 
\[
W_{L^2_H(\d \rho_j)}=\sqrt{\rho_j}W_{L^2_H(\d \mu)}.
\] 
Using the same white noise in  $L^2_H(\d \mu)$ we can couple two fields in such a way that their difference is a stationary field with 
\[
\sigma^2=\sup_{|\alpha|\le k+1}(2\pi)^{2|\alpha|}\|t^\alpha (\sqrt{\rho_2}-\sqrt{\rho_1})(t)\|_{L^2_H(\d \mu)}^2.
\]
Since it does not really matter what is $\mu$, this method is applicable to any two spectral measures $\mu_1$ and $\mu_2$. We can always take $\d\mu=\d \mu_1+\d \mu_2$ and consider the corresponding $\rho_j=\d \mu_j/\d\mu$. But if measures $\mu_1$ and $\mu_2$ are too singular with respect to each other we do not have small $\sigma$. In the extreme case of singular measures we have that $\rho_j$ are characteristic functions of disjoint sets and the coupling is the same as the independent coupling.

\section{General case}
\label{s: general case}
The coupling constructed in the previous section is very natural and gives a very good estimate on the tightness of fields. Unfortunately, there are important examples that are not covered by this construction. There are two main reasons why it might be not applicable. First of all, it might be that there is no spectral density. In particular, this happens in the important case of the so called monochromatic waves (random plane wave and arithmetic waves), for these models the support of the spectral measure is either one-dimensional or a finite set of points. The corresponding fields are not in $L^2$ and can't be written as a convolution with the white noise.

There is a standard way to deal with this problem: one can convolve the spectral measure with a mollifier, equivalently, multiply the covariance function by a function which is fast decaying and close to $1$ on a large compact. Unfortunately, this does not immediately leads to a tight coupling (but the results of this section will give an explicit construction). Moreover, it still does not cover the following case. Let us assume that spectral measures are orthogonal: in this case, even if there are spectral densities, the coupling above is the trivial independent coupling, which is not tight. One might argue, that in the case of orthogonal spectral measures one should not expect close couplings, but there are situations where spectral measures are still close in some sense and there should be some kind of close coupling. 

To explain the idea, let us consider the simplest example which is similar to coupling from \cite{BeMa}. Let $\rho_1=\rho_{x_1}=(\delta_{x_1}+\delta_{-x_1})/2$ be the sum of two symmetric delta measures. In this case 
the field is 
\[
f_1(x)=b\cos(2\pi x\cdot x_1)-c\sin(2\pi x\cdot x_1),
\]
where $b$ and $c$ are independent real $N(0,1)$ random variables. Its covariance is 
\[
\E f(x)f(y)=\cos(2\pi (x-y)\cdot x_1).
\]

Now consider another field like this with spectral measure $\rho_2=\rho_{x_2}$. It is clear that $\rho_2\to\rho_1$ in weak-* topology as $x_2\to x_1$. In particular, $f_2$ converges to $f_1$ in distribution. On the other hand, it is easy to see that if $(x-y)\cdot (x_1-x_2)$ is not small, the difference between covariance kernels is not small, and it is impossible to have a coupling which has small variance everywhere. Kernels are close inside a disc of radius $R$ as long as $|x_1-x_2|=o(1/R)$. 

The most natural coupling in this case is to write
\[
f_2(x)=b\cos(2\pi x\cdot x_2)-c\sin(2\pi x\cdot x_2),
\]  
where $b$ and $c$ are the same as in the definition of $f_1$. With this coupling, the difference is small in a disc of radius $R=o(|x_2-x_1|^{-1})$ as long as the coefficients are not too large. The last event has Gaussian tails.  

If we have a converging sequence of spectral measures of this type, then we can couple all of them in such a way that they are getting closer and closer in larger and larger domains.

The goal of this section is to generalise this construction to the case of general measures.

As explained above, a field is a functional of the white noise, hence a coupling of  fields can be given by a coupling of white noises. In this section we assume that all fields have variance $1$, equivalently, spectral measures have total mass $1$ (otherwise we just rescale the fields). In this case, we can couple white noises by coupling spectral measures.

Let $\rho_1$ and $\rho_2$ be two spectral measures on $\R^n$ and $\rho$ be their symmetric coupling on $\R^n\times \R^n$, i.e. a symmetric measure such that its marginals are $\rho_1$ and $\rho_2$. Let $W$ be the white noise in $L^2_H(\R^n\times\R^n, \rho)$. For $f\in L^2_H(\R^n,\rho_1)$ we can define $W_1(f)=W(f)$ where we extend $f$ from $\R^n$ to $\R^{2n}$ by $f(s,t)=f(s)$. Essentially, $W_1$ is the projection of $W$ onto the first coordinate. It is easy to see that in this case
\[
\begin{aligned}
\E(W_1(f)W_1(g))=&\E(W(f)W(g))=\int_{\R^n\times\R^n} f(s)\bar{g}(s)d\rho(s,t)
\\
=&\rho_2(\R^n)\int_{\R^n} f(s)\bar{g}(s)d \rho_1(s). 
\end{aligned}
\]
Since $\rho_2$ is a probability measure we see that $W_1$ is the white noise in $L^2(\R^n,\rho_1)$. In the same way we can define $W_2$. This constructions shows that any coupling of measures gives rise to a coupling of white noises and fields.  To be more precise, we have
\[
\begin{aligned}
f_1(x)=&W_1(e^{2\pi i x\cdot s}) =W(e^{2\pi i x\cdot s})= W(e^{2\pi i (x,0)\cdot(s,t)}),
\\
f_2(x)=&W_2(e^{2\pi i x\cdot t})= W(e^{2\pi i x\cdot t})=W(e^{2\pi i (0,x)\cdot(s,t)}).
\end{aligned}
\]
Their difference is 
\[
F(x)=W(e^{2\pi i x\cdot s}-e^{2\pi i x\cdot t}).
\]
This is a Gaussian field with covariance
\[
\E F(x)F(y)=\int_{\R^{n}\times\R^n} (e^{2\pi i x\cdot s}-e^{2\pi i x\cdot t})(e^{-2\pi i y\cdot s}-e^{-2\pi i y\cdot t})\d \rho(s,t).
\]
First of all, we note that, as expected, this field is not stationary. In particular,
\[
\Var F(x)=\int_{\R^{n}\times\R^n}\left|e^{2\pi i x\cdot s}-e^{2\pi i x\cdot t}\right|^2 \d\rho(s,t)
\]
does depend on $x$ (note that $\Var F(0)=0$).

In the same way, by differentiating the covariance kernel, for any multi-index $\alpha$ with $|\alpha|\le k+1$ we have
\begin{equation}
\label{eq:var d F}
\Var \partial_x^\alpha F=
 (2\pi)^{2|\alpha|}\int_{\R^{n}\times\R^n}\left|t^\alpha e^{2\pi i x\cdot s}-s^\alpha e^{2\pi i x\cdot t}\right|^2 \d\rho(s,t).
\end{equation}

According to Lemma \ref{l:Borell-TIS}, the norm of $F$ in $\Omega$ is controlled by the supremum of these variances over all $x\in \Omega$. It is not hard to write a uniform upper bound on the integrand: 
\[
\Var \partial_x^\alpha F\le C(|x|^2+1)\int (|s|^2+|t|^2+1)^{k+1}|s-t|^2\d \rho(s,t)
\]
where $C$ is a constant which depends on the dimension and $k$, but not on $\rho$ or $x$.  By optimizing the coupling we have the following theorem.
\begin{theorem}
\label{thm: general coupling}
Let $f_1$ and $f_2$ be two stationary Gaussian fields with spectral measures $\rho_1$ and $\rho_2$. We assume that these measures have finite moments of order up to $2(k+1)$ for some integer $k$. In this case there is a coupling such that $F=f_2-f_1$ is a non-stationary Gaussian field such that 
\begin{equation}
\label{eq:sigma for F}
\begin{aligned}
\sigma^2_R&=\sup_{x\in B_{R+1}}\sup_{|\alpha|\le k+1} \Var \partial_x^\alpha F(x)
\\
&
\le
 C(R^2+1)\inf_{\rho\in \mathcal{P}(\rho_1,\rho_2)}\int (|s|^2+|t|^2+1)^{k+1}|s-t|^2\d \rho(s,t),
\end{aligned}
\end{equation}
where $\mathcal{P}(\rho_1,\rho_2)$ is the space of all symmetric couplings of $\rho_1$ and $\rho_2$ and $C$ is a constant which depends only on $n$ and $k$.
\end{theorem}

\begin{remark} The infimum in \eqref{eq:sigma for F} is the Wasserstain distance between measures $\rho_1$ and $\rho_2$ with the cost function $(|s|^2+|t|^2+1)^{k+1}|s-t|^2$. Note that this is very similar to the square of the Wasserstein-$2$ distance between moments of $\rho_1$ and $\rho_2$. 
\end{remark}

\section{Applications}
\label{s: application}
\subsection{Arithmetic plane waves}
\label{ss: arithmetic wave}
There are many scenarios where one would like to show that a sequence of fields converge to a limiting field not only in distribution, but also uniformly (with large probability). In this section we would like to discuss a particular case of arithmetic plane waves. 

Let us fix the dimension $n\ge 2$ and let $\S^{n-1}$ be the unit ball in $\R^{n}$. The \emph{arithmetic random  wave} is a Gaussian field in $\R^n$ such that its spectral measure is 
\[
\rho_m=\frac{1}{r_n(m)}\sum_{\lambda\in \Lambda_m}\delta_\lambda,
\]
where $\Lambda_m=\{\lambda\in \Z^n/\sqrt{m}: |\lambda|=1\}$ and $r(m)=r_n(m)=|\Lambda_m|$ is the cardinality of $\Lambda_m$ (equal to the number of ways to write $m$ as a sum of $n$ squares). 
This field is a standard Gaussian function $f_m$ in the space of eigenfunctions of the Laplacian with eigenvalue $4\pi^2$ in the flat torus of size $\sqrt{m}$.  We are interested in the behaviour of this field as $m$ tends to infinity. 

\subsubsection{Dimension $n=2$}
When the dimension $n=2$ the field is not  quite well defined, since some values $m$ can not be represented as the sum of two squares. For these values $\Lambda_m=\emptyset$ and the field is not defined. In fact, this happens for most of $m$. Let $S_2$ be the set of $m$ which are sums of two squares. Landau \cite{landau} proved in 1908 that the set $S_2$ has zero density, to be more precise $|S_2\cap [0,n]|\approx c n/\sqrt{\log n}$. In this paper we always implicitly assume that $m\in S_2$ and for various subsets of $S_2$ we will discuss their densities relative to $S_2$. Namely, we say that a sequence $m_k$ is \emph{generic} or of \emph{density $1$} if $|\{m_k\le n\}|/|\{m\in S_2, m\le n\}|\to 1$ as $n\to \infty$. But even within $S_2$ the behaviour of $r(m)$ and of $\rho_m$ is highly irregular. For example, $r(2^n)=4$ for every $n$, but $r(m)$ is unbounded. It is known that $r(m)=o(m^\epsilon)$ for every $\epsilon>0$. On the other hand, there is $c>0$ such that $\log \log m \le c \log r(m)$ for a density $1$ subsequence. It is also known, that for generic $m$ the measure $\rho_m$ converges weakly to the normalized arc-length, but there are (zero density) sequences such that $\rho_m$ converges to other measures. See  \cite{equidi,erdhal,krkuwi,kurwig} for details.

In the case when $\rho_{m_k}$ converges to the uniform distribution, the field $f_{m_k}$ converges to the random plane wave in distribution. If we can quantify how fast the measures converge in the sense of \eqref{eq:sigma for F}, then Theorem \ref{thm: general coupling} and Corollary \ref{cor: norm of the field} give us the coupling such that the fields converge in $C^k$ norm in a growing ball.

\subsubsection{Dimensions $n=3$ and $n\ge 4$}
In the case of $n=3$ more integers belong to the corresponding set $S_3$ and the distribution of integer points on the sphere is more uniform. To be more precise, $S_3$ contains all integers that are not of the form $4^a(8b+7)$. Unlike the dimension $n=2$, this set has positive density equal to $5/6$ \cite{Wagstaff}. Multiplying $m$ by a power of $4$ does not change the number of points, hence it is natural to redefine $S_3$ and to consider only $m$ that are not equal to $0$, $4$ or $7$ mod $8$. For every $\epsilon>0$ and for $m\in S_3$ we have that $m^{1/2-\epsilon}=o(r_3(m))$ and $r_3(m)=o(m^{1/2+\epsilon})$. The set $\Lambda_m$ is uniformly distributed in $\S^2$. The precise statement will be given later, but we have that $\rho_m$ converges to the uniform measure and the field converges in the law to the random plane wave which is the field with the spectral measure which is the uniform measure on $\S^2$. 

In higher dimension the distribution becomes even more regular. First of all every number can be written as a sum of at least four integers. In dimension $n=4$ there are some exceptional values with small $r_4(m)$. In particular, $r_4(2^a)=24$. If $n\ge 5$ (or $n=4$ but we exclude $m$ with small $r(m)$), then $r_n(m)$ is of order $m^{(n-2)/2}$. Moreover, the points of $\Lambda_m$ are asymptotically uniformly distributed. 


Our goal is to show that arithmetic waves converge to the random plane wave not only is law, but we can couple them so that the difference converges locally uniformly. Moreover, we want to quantify the rate of convergence.

First, we note that the measures are singular with respect to each other, this means that strong results of Section \ref{s: smooth case} are not applicable, so we are going to use Theorem \ref{thm: general coupling}.  Next, since $\rho_m$ and the surface measure are supported on $\S^n$, for $R>1$ we can estimate $\sigma_R^2$ from \eqref{eq:sigma for F} by
\begin{equation}
\label{eq: arithmetic sigma}
\sigma_R^2\le CR^2W_2^2(\rho_m,\sigma),
\end{equation}
where $\sigma$ is normalized surface area and 
\[
W_2^2(\rho_m,\sigma)=\inf_{\rho\in \mathcal{P}(\rho_m,\sigma)}\int |t-s|^2\d \rho(t,s)
\]
is the Wasserstein (transport) distance between $\rho_m$ and $\sigma$, and $C$ is a constant that might depend on the dimension $n$ and degree $k$, but not on $m$.

In order to estimate the Wasserstein distance we are going to use a very simple observation based on the transport interpretation of the distance.
\begin{proposition}
\label{prop: W_2 and discrepancy}
Let $\Omega_j$ be a partition of $\S^n$ such that the diameter of each $\Omega_j$ is at most $r$. For two probability measures $\rho_1$ and $\rho_2$ on $\S^n$ we define 
\begin{equation}
\label{eq: discrepancy def}
\Delta(\Omega)=\Delta_{\rho_1,\rho_2}(\Omega)=|\rho_1(\Omega)-\rho_2(\Omega)|
\end{equation}
to be the discrepancy of these measures on $\Omega$. Then
\[
W_2^2(\rho_1,\rho_2)\le C (r^2+\sum_j\Delta(\Omega_j))
\]
where $C$ is an absolute constant.
\end{proposition}
\begin{proof}
Within each $\Omega_j$ we can transport (couple) smaller measure to a part of the larger one. Since in this coupling we do not move points by more than $r$, the total transport cost is at most $r^2\min(\rho_1(\Omega_j),\rho_2(\Omega_j))$. Summing over $j$ we have the upper bound $r^2$. We don't have control over the transport of the remaining part of the measures, but since $\S^n$ is bounded we have an upper bound of $C \sum \Delta(\Omega_j)$. Summing two estimates we prove the proposition.
\end{proof}


Study of the distribution of integer points on spheres is a classical subject and the estimates of the discrepancy between $\rho_m$ and $\sigma$ are already available. They are of slightly different form in small and large dimensions. 

\begin{proposition}[Dimension $n=2$: Erd\"os-Hall \cite{erdhal}]
\label{erdhalprop}
For every $\epsilon>0$ there is a density one subsequence of $S_2$ such that 
\begin{equation*}
\Delta_{\rho_m,\sigma}(I)<{\log(m)^{-\kappa+\epsilon}}
\end{equation*}
where $\kappa=\log(\pi/2)/2$ and $I$ is any arc in $\S^1$.
\end{proposition}


\begin{proposition}[Dimension $n=3$: Duke and Schulze-Pillot \cite{DSP}]
\label{prop: DSP}
Let $n=3$, $m\in S_3$ and $\Omega$ be a convex domains on $\S^{2}$ with piecewise smooth boundary.   Then, for every $\epsilon>0$
\[
\Delta_{\rho_m,\sigma}(\Omega)=O(m^{-\frac{1}{175}+\epsilon}).
\]
\end{proposition}

\begin{proposition}[{Dimension $n\ge 4$: Malyshev  \cite[Main Theorem]{Malyshev} and Podsypanin \cite{Podsypanin}}]
\label{prop: Malyshev}
Let $n\ge 4$ and $\Omega_m$ be convex domains on $\S^{n-1}$ with piecewise smooth boundary.   Then, for every $\epsilon>0$
\[
\Delta_{\rho_m,\sigma}(\Omega_m)=O(m^{-\frac{n-3}{4(3n-2)}+\epsilon}),
\]
where the implicit constant on $O$ is independent of $\Omega$ and $m\in S_n$.
\end{proposition}

Given these estimates we immediately use Proposition \ref{prop: W_2 and discrepancy} to obtain the following theorem.
\begin{theorem}
\label{thm:arithmetic coupling}
Let $n\ge 2$ be the dimension, $k$ be the degree of smoothness, and $\epsilon>0$ be an arbitrary constant. Let us consider $m\to \infty$ such that $m\in S_n$ where $S_n$ is the collection of $m$ such that $r_n(m)\to \infty$. Then there is a coupling between $\rho_m$ and $\sigma$ such that for every $R>1$
\[
\sigma_R^2\le 
\begin{cases}
CR^2(\log m)^{-2\kappa/3+\epsilon} & n=2,\\
CR^2m^{-1/350+\epsilon}& n= 3,\\
CR^2m^{-\frac{n-3}{2(3n-2)(n+1)}+\epsilon}& n \ge 4,
\end{cases}
\]
where $\kappa$  as in Propositions \ref{erdhalprop} and $C$ is a constant independent of $m$ and $\epsilon$.
\end{theorem}

\begin{proof}
This theorem follows immediately from the discrepancy estimates and Proposition \ref{prop: W_2 and discrepancy}. Indeed, for a small $r>0$ one can partition $\S^{n-1}$ into domains $\Omega_j$ in such a way that each $\Omega_j$ has piecewise smooth boundary,  diameter comparable to $r$ and such that their number is comparable to $r^{-n+1}$. In case $n=2$, $\Omega_j$ are just arcs on $\S^1$. By  Propositions \ref{prop: W_2 and discrepancy}, \ref{erdhalprop}, \ref{prop: DSP} and \ref{prop: Malyshev} we have that 
\[
W_2^2(\rho_m,\sigma)\le C (r^2+\sum_j \Delta(\Omega_j))
\le 
\begin{cases}
C(r^2+r^{-1}(\log m)^{-\kappa+\epsilon}) & n=2,\\
C(r^2+r^{-2}m^{-1/175+\epsilon})& n= 3,\\
C(r^2+r^{-n+1}m^{-\frac{n-3}{4(3n-2)}+\epsilon})& n\ge 4.
\end{cases}\]
Taking $r=(\log m)^{-\kappa/3}$ if $n=2$, $r=m^{-1/175\cdot 4}$ if $n=3$ and $r=m^{-\frac{n-3}{4(3n-2)(n+1)}}$ if $n\ge 4$ and recalling \eqref{eq: arithmetic sigma} we complete the proof.
\end{proof}

In dimension $d=3$ we can improve the result of Theorem \ref{thm:arithmetic coupling} if we assume the following conjecture which is an extension of Bourgain-Sarnak-Rudnick's result \cite[Theorem 1.7]{bsr016} for shrinking spherical caps and segments to more general shrinking regions.
\begin{conjecture}
\label{theconj}
Assume that $n=3$ and $m\in S$. Let $A_m$ be a family of sets on the sphere which depend on $m$. We assume that $A_m$ is almost round in the sense that there are absolute constants $c_1$ and $c_2$ such that $A_m$ contains a spherical cap of radius $c_1\sigma(A_m)^{1/2}$ and is contained in a spherical cap of radius $c_2\sigma(A_m)^{1/2}$. We also assume that $r(m)^{-1+\delta}\lesssim \sigma(A_m)\lesssim r(m)^{-\delta}$ for some $\delta>0$.  Then for every $\epsilon>0$
\[
    \int_{SO(3)}\Delta^2(gA_m)dg=O(m^\epsilon r(m)^{-1}\sigma(A_m)).
\]
\end{conjecture}

For each $m$ one can partition $S^2$ into domains $A_j$ such that they all are `almost round' and their area is of the same order $r^2$. Then we have 
\[
\int_{SO(3)}\sum_j\Delta^2(gA_j)dg=O(m^\epsilon r(m)^{-1}).
\]
This implies that there is $g\in SO(3)$ such that 
\[
\sum_j\Delta^2(gA_j)dg=O(m^\epsilon r(m)^{-1}).
\]
Without loss of generality we can assume that $g$ is the identity (otherwise instead of the original partition we take the rotated one). This means that there is a partition such that the sum of squared discrepancies is small. 

Then we have
\[
\sum \Delta(A_j)\le \br{\sum \Delta^2(A_j)\sum 1}^{1/2}=O(m^{\epsilon/2}r(m)^{-1/2}r^{-1})=O(m^{-1/4+\epsilon}r^{-1}).
\]
As before, we have to optimize
\[
r^2+\sum\Delta(A_j)=r^2+O(m^{-1/4+\epsilon}r^{-1}).
\]
So we take $r=m^{-1/12}$ (which corresponds to $\sigma=m^{-1/6}$ and satisfies the assumption of Conjecture \ref{theconj}). This leads to
\[
W_2^2(\rho_m,\sigma)\le Cm^{-1/6+\epsilon}
\]
and
\[
\sigma^2_R\le CR^2m^{-1/6+\epsilon}.
\]

\bibliography{coupling}{}
\bibliographystyle{plain}



\end{document}